\documentclass{monsky2009}

\usepackage{amssymb}

\usepackage{bm}

\newenvironment{conjecture*}[1][]{\textbf{Conjecture #1\hspace{.3em}}}{}
\newenvironment{theorem*}[1]{\textbf{#1}\itshape \hspace{.3em}}{\upshape}
\newenvironment{remark*}[1]{\textbf{#1}\itshape \hspace{.3em}}{\upshape}
\newenvironment{proof}{\textbf{Proof\hspace{.3em}}}{}

\newenvironment{proofsketch}{\textbf{Proof Sketch\hspace{.3em}}}{}

\newtheorem{definition}{Definition}[section]
\newtheorem{theorem}[definition]{Theorem}
\newtheorem{lemma}[definition]{Lemma}

\newtheorem{corollary}[definition]{Corollary}

\newcounter{kpremark}
\newenvironment{remark}{\addtocounter{kpremark}{1}\textbf{Remark \thekpremark}\itshape \hspace{.3em} }{\upshape}

\newcommand{\charstic}{\ensuremath{\mathrm{char\ }}}
\newcommand{\ord}{\ensuremath{\mathrm{ord}}}

\newcommand{\mod}[1]{\ensuremath{\quad(#1)}}

\newcommand{\rgen}{\ensuremath{R_{\mathit{gen}}}}
\newcommand{\funnysum}{\ensuremath{\sum_{\mbox{\parbox[t]{2.5em}{$\scriptstyle rs=Q$\\[-1.3ex] $\scriptstyle s\ne Q$}}}}\hspace{-.6em}}
\newcommand{\acs}{\qquad} % array column space

\begin{document}

\begin{frontmatter}



%\title{}


%\author{}

\title{Tight closure's failure to localize---a self-contained exposition}
\author{Paul Monsky}

\address{Brandeis University, Waltham MA  02454-9110, USA. monsky@brandeis.edu}

\begin{abstract}
We give a treatment of the Brenner-Monsky example based on polynomial algebra and linear algebra. No prior knowledge of tight closure theory, Hilbert-Kunz theory, algebraic geometry or local cohomology is assumed.

%\vspace{2ex}
\end{abstract}

%\begin{keyword}

%\PACS 
%\end{keyword}
\end{frontmatter}

% main text

\section{Introduction---Brenner's insight}
\label{section1}

Holger Brenner and I have given a negative solution to the localization problem for tight closure \cite{1}. The argument involves the Hilbert-Kunz theory of plane curves (and in particular \cite{2}) together with results of Brenner, Hochster and Huneke on test elements and local cohomology.

But most of this machinery, useful as it is for understanding our counterexample, may be dispensed with; in this paper I give a treatment of the example, using only linear algebra and a little local cohomology developed ab initio. The reader doesn't need to know anything about Hilbert-Kunz theory, homological algebra, vector bundles or tight closure. Though the arguments are largely drawn from \cite{1} and \cite{2}, everything is proved here from scratch.

\begin{definition}
\label{def1.1}
If $A$ is a Noetherian domain of characteristic $p>0$, $q$ is a power of $p$, and $I$ is an ideal of $A$, $I^{[q]}$ is the ideal generated by all $v^{q}$, $v$ in $I$.
\end{definition}

\begin{definition}
\label{def1.2}
$u$ is in the tight closure, $I^{*}$, of $I$ if for some $d \ne 0$, $du^{q} \in I^{[q]}$ for all $q$.
\end{definition}

Suppose now that $S \subset A$ is multiplicatively closed, $0 \not\in S$. Then $S^{-1}I$ is an ideal of $S^{-1}A$ and we can form the ideal $(S^{-1}I)^{*}$. The localization problem asks whether $(S^{-1}I)^{*}$ is always equal to $S^{-1}(I)^{*}$. In other words, suppose that $f \in (S^{-1}I)^{*}$. Must there exist an $s$ in $S$ such that $sf \in I^{*}$? After giving positive solutions to the localization problem in some special cases, Brenner realized that the study of a 1-parameter family would give a negative answer provided the family satisfied a certain counter-intuitive condition. I'll explain this insight of Brenner's in the context of a 1-parameter family of projective plane curves. Let $L$ be algebraically closed of characteristic $p$, and let $P$ and $P_{1}$ in $L[x,y,z]$ be homogeneous of the same degree. For $\alpha$ in $L$ set $g_{\alpha}=P+\alpha P_{1}$ and $R_{\alpha}=L[x,y,z]/g_{\alpha}$. $\rgen$ is the ring $L(t)[x,y,z]/P+tP_{1}$. Fix $f$ in $L[x,y,z]$, and an ideal $I$ in $L[x,y,z]$. $I$ generates ideals in each $R_{\alpha}$ and in $\rgen$; abusing language we call all these ideals $I$.

\begin{theorem}[Brenner]
\label{theorem1.3}
 Suppose that:
\begin{enumerate}
\item[(a)] $f \in I^{*}$ in $\rgen$
\item[(b)] There exist infinitely many $\alpha$ in $L$ for which $R_{\alpha}$ is a domain and $f \not\in I^{*}$ in $R_{\alpha}$
\end{enumerate}

Then the localization problem has a negative answer.
\end{theorem}

\begin{proof}
Take $A=L[x,y,z,t]/P+tP_{1}$, and let $S \subset A$ be $L[t]-\{0\}$. Note that $S^{-1}A$ and $A/(t-\alpha)$ identify with \rgen\ and $R_{\alpha}$ respectively. $I \subset L[x,y,z]$ generates an ideal in $A$ that we again call $I$. Since \rgen identifies with $S^{-1}A$, (a) tells us that $f\in (S^{-1}I)^{*}$ in $S^{-1}A$. Suppose however that $sf\in I^{*}$ for some $s=s(t)$. Then for some $d\ne 0$ in $A$, $ds^{q}f^{q}\in I^{[q]}$ for all $q$. Now by (b) there are infinitely many $\alpha$ in $L$ with $A/(t-\alpha)$ a domain and $f\not\in I^{*}$ in $A/(t-\alpha)$. The corresponding ideals, $(t-\alpha)$, are distinct height 1 primes in $A$, and so cannot all contain $ds$. Fix one such $t-\alpha$ with $ds\not\in (t-\alpha)$. If $\bar{d}$ is the image of $d$ in $A/(t-\alpha)=R_{\alpha}$, then $\bar{d}s(\alpha)^{q}f^{q}\in I^{[q]}$ in $R_{\alpha}$ for all $q$. But $\bar{d}\ne 0$ and $s(\alpha)$ is a non-zero element of $L$. We conclude that $\bar{d}f^{q}\in I^{[q]}$ in $R_{\alpha}$ for all $q$, contradicting the choice of $\alpha$.
\qed
\end{proof}

How is a 1-parameter family satisfying (a) and (b) to be found? In \cite{2}, I had studied a 1-parameter linear family of characteristic 2 plane quartics, obtaining counter-intuitive results suggestive of (a) and (b). (This was done in ignorance of tight closure; my goal was to calculate the ``Hilbert-Kunz multiplicities'' of the curves in this family.)  It turned out that the matrix calculations in \cite{2}, slightly extended and combined with a suitable ``test element theorem'', were exactly what was needed to produce the example. In the following three sections I describe these calculations. The final two sections use some simple algebra to complete the proof.

Throughout, $L$ will be a field of characteristic 2, and $P$ the element $z^{4}+xyz^{2}+(x^{3}+y^{3})z$ of $L[x,y,z]$. If $\alpha \in L$, $\alpha \ne 0$, $g_{\alpha}=P+\alpha x^{2}y^{2}$. It's easy to see that $g_{\alpha}$ is irreducible, so that $R_{\alpha} = L[x,y,z]/g_{\alpha}$ is a domain. Fix a power, $Q$, of 2. $O$ will be the graded $L$-algebra $L[x,y,z]/(x^{4Q},y^{4Q},z^{4Q})$. Multiplication by $g_{\alpha}$ gives a map $O_{j}\rightarrow O_{j+4}$ for each $j$.  The key to establishing (a) and (b) of Theorem \ref{theorem1.3} is the close study of the kernel $N_{6Q-5}$ of $g_{\alpha}: O_{6Q-5}\rightarrow O_{6Q-1}$, both when $\alpha$ is transcendental over $Z/2$ and $Q\ge 2$, and when $\alpha$ is algebraic over $Z/2$ and $Q$ is a certain power of 2 attached to $\alpha$. When $Q\ge 2$, $O_{6Q-5}$ and $O_{6Q-1}$ have dimensions $12Q^{2}-12$ and $12Q^{2}$ and one might expect $N_{6Q-5}=(0)$ for every choice of $Q$. This is true for transcendental $\alpha$ (see Theorem \ref{theorem3.13}) but false for algebraic $\alpha$ (Corollary \ref{corollary4.6}, with $Q$ as in Definition \ref{def4.4}).

\section{Some identities involving $\bm{P}$}
\label{section2}

We begin by defining some elements of $Z/2[x,y]$.

\begin{definition}
\label{def2.1}
If $r$ is a power of 2 then:

\begin{enumerate}
\item[(1)] $A_{r}$ (resp.\ $B_{r}$) is $\sum x^{i}y^{j}$, the sum extending over all pairs $(i,j)$ with $i\equiv j \mod{3}$ and $i+j=4r-2$ (resp.\ $4r-1$).
\item[(2)] $C_{1}=1$ and $C_{2r}=A_{r}^{2}$.
\end{enumerate}
\end{definition}

Each monomial $x^{i}y^{j}$ appearing in $A_{2r}$ has $i\equiv j \mod{2}$. Those monomials with $i$ (and $j$) even sum to $B_{r}^{2}$, while those with $i$ (and $j$) odd sum to $xyA_{r}^{2}$. So $A_{2r}=B_{r}^{2}+xyA_{r}^{2}$.  A similar argument shows that $B_{2r}=(x^{3}+y^{3})A_{r}^{2}$.

\begin{lemma}
\label{lemma2.2}
The following identities hold in $Z/2[x,y,z]$: When $Q$ is a power of 2, $z^{4Q}=A_{Q}z^{2}+B_{Q}z+\sum_{rs=Q}(C_{r}P)^{s}$.
\end{lemma}

\begin{proof}
Since $A_{1}=xy$, $B_{1}=x^{3}+y^{3}$ and $C_{1}=1$, the case $Q=1$ follows from the definition of $P$. In general we argue by induction, squaring the identity for $Q$, replacing $z^{4}$ by $xyz^{2} + (x^{3}+y^{3})z+P$, and using the identities following Definition \ref{def2.1}.
\qed
\end{proof}

Now let $L$, $Q$ and $O=L[x,y,z]/(x^{4Q}, y^{4Q}, z^{4Q})$ be as in the last section.

\begin{definition}
\label{def2.3}
$R_{Q}$ is the element $A_{Q}z^{2}+B_{Q}z$ of $O_{4Q}$. $\Delta$ in $O_{6Q-1}$ is $\sum x^{i}y^{j}z^{k}\!$, the sum extending over all triples $(i,j,k)$ with $i+j+k=6Q-1$, $i \not\equiv j \mod{3}$ and $k=1$ or 2.
\end{definition}

\begin{lemma}
\label{lemma2.4}
Suppose $i+j=2Q-1$. Then, in $O$, $(x^{i}y^{j}+x^{j}y^{i})R_{Q}$ is 0 if $i\equiv j \mod{3}$ and is $\Delta$ otherwise.
\end{lemma}

\begin{proof}
Definition \ref{def2.1} shows that $(x^{3}+y^{3})A_{Q}$ and $(x^{3}+y^{3})B_{Q}$ both lie in  $(x^{4Q},y^{4Q})$. So $x^{3}R_{Q}=y^{3}R_{Q}$ in $O$. It follows immediately that when $i+j=2Q-1$ then $x^{i}y^{j}R_{Q}$ only depends on $i \bmod{3}$. This gives the first part of Lemma \ref{lemma2.4} and shows that when $i \not\equiv j \mod{3}$, $(x^{i}y^{j}+x^{j}y^{i})R_{Q}=(x^{Q-1}y^{Q}+x^{Q}y^{Q-1})R_{Q}$. But this last element is easily seen to be $\Delta$.
\qed
\end{proof}

\pagebreak
\begin{theorem}
\label{theorem2.5}
\hspace{2em} %so there's a linebreak

\begin{enumerate}
\item[(1)] In $O$, $R_{Q}=\sum_{rs=Q}(C_{r}P)^{s}$
\item[(2)] Suppose that $i+j=2Q-1$. Then in $O$, $(x^{i}y^{j}+x^{j}y^{i})P^{Q}=\varepsilon\Delta+(x^{i}y^{j}+x^{j}y^{i})\funnysum(C_{r}P)^{s}$, where $\varepsilon$ is 0 if $i\equiv j \mod{3}$ and is 1 otherwise.
\end{enumerate}
\end{theorem}

\begin{proof}
Combining Lemma \ref{lemma2.2} with the definition of $R_{Q}$, noting that $z^{4Q}=0$ in $O$, we get (1). Since $C_{1}=1$, $P^{Q}=R_{Q}+\funnysum (C_{r}P)^{s}$. Multiplying by $x^{i}y^{j}+x^{j}y^{i}$ and applying Lemma \ref{lemma2.4} gives (2).
\qed
\end{proof}

\begin{lemma}
\label{lemma2.6}
Suppose $i+j=2Q-1$. The co-efficient of $x^{4Q-2}y^{Q-2}$ in $(x^{i}y^{j}+x^{j}y^{i})(x^{3}+y^{3})^{Q-1}$ is 0 if $i\not\equiv j \mod{3}$ and is 1 otherwise.
\end{lemma}

\begin{proof}
The first assertion is clear. For the second note that the co-efficient in question is the sum of the co-efficients of $x^{4Q-2-i}y^{Q-2-j}$ and $x^{4Q-2-j}y^{Q-2-i}$ in $(x^{3}+y^{3})^{Q-1}=x^{3Q-3}+x^{3Q-6}y^{3} + \cdots + y^{3Q-3}$.  The first of these co-efficients is 1 when $i$ is both $\ge Q+1$ and $\equiv j \mod{3}$, while the second is 1 when $j$ is both $\ge Q+1$ and $\equiv i\mod{3}$. Since precisely one of $i$ and $j$ is $\ge Q+1$ (they cannot be $Q$ and $Q-1$) we get the lemma.
\qed
\end{proof}

\section{The spaces $\bm{X}$ and $\bm{Y}$---the case of transcendental $\bm{\alpha}$}
\label{section3}

$L$, $P$, and $g_{\alpha}$ are as in the final paragraph of the introduction.  $Q\ge 2$ is a power of 2. $O$ is the graded $L$-algebra $L[x,y,z]/(x^{4Q},y^{4Q},z^{4Q})$, and $N_{6Q-5}$ is the kernel of $g_{\alpha}: O_{6Q-5}\rightarrow O_{6Q-1}$.

\begin{definition}
\label{def3.1}
\hspace{2em} % so there will be a linebreak

\begin{enumerate}
\item[(1)] $[i,j]=x^{i}y^{j}+x^{j}y^{i}$
\item[(2)] $X \subset O_{6Q-5}$ is spanned by the $[i,j]P^{k}$ with $i+j+4k=6Q-5$ and $k=0,1,\ldots ,Q-1$
\item[(3)] $Y \subset O_{6Q-1}$ is spanned by the $[i,j]P^{k}$ with $i+j+4k=6Q-1$ and $k=0,1,\ldots ,Q-1$
\end{enumerate}
\end{definition}

\begin{theorem}
\label{theorem3.2}
Let $(Y,\Delta)$ be the subspace of $O_{6Q-1}$ spanned by $Y$ and the element $\Delta$ of Definition \ref{def2.3}. Then $g_{\alpha}\cdot X \subset (Y,\Delta)$.
\end{theorem}

\begin{proof}
Evidently $(x^{2}y^{2})\cdot X \subset Y$. It remains to show that $P\cdot X \subset (Y,\Delta)$. This will follow if we can prove that $P\cdot [i,j]\cdot P^{Q-1}\in (Y,\Delta)$ whenever $i+j=2Q-1$. By Theorem \ref{theorem2.5} it suffices to show that each $[i,j]C_{r}^{s}P^{s}$ is in $Y$ when $rs=Q$ and $s<Q$. This is easy: $[i,j]\cdot C_{r}^{s}$ is a symmetric form in $x$ and $y$ of (odd) degree $(2Q-1)+s(4r-4)=6Q-1-4s$.
\qed
\end{proof}

The $[i,j]P^{k}$ with $i+j+4k=6Q-5$, $i,j< 4Q$, $k<Q$, and $i$ odd evidently span $X$. Noting that each such element has the form $(x^{i}y^{j}+x^{j}y^{i})z^{4k}+$ terms of lower degree in $z$, with $i$ odd and $j$ even, we see that these elements are a basis of $X$. One constructs a basis of $Y$ similarly and finds that $\dim X=\dim Y$; both dimensions are in fact $\frac{3Q^{2}}{2}$. A basis of $(Y,\Delta)$ is given by the $[i,j]P^{k}$ with $i+j+4k=6Q-1$, $i,j<4Q$, $k<Q$, $i$ odd, together with $\Delta$.

Note that the kernel of the map $g_{\alpha}: X\rightarrow (Y,\Delta)$ of Theorem \ref{theorem3.2} is just $N_{6Q-5}\cap X$. We'll get a better understanding of this space by replacing $X$ and $(Y,\Delta)$ by certain quotients.

\begin{definition}
\label{def3.3}
$D$ is the graded $L$-algebra $L[x,y]/(x^{4Q},y^{4Q})$. $E_{i}$, $1\le i \le Q$, is the element $[2i-1,2Q-2i]$ of $D_{2Q-1}$. $F_{i}$, $1\le i \le Q$, is the element $x^{2Q}y^{2Q}E_{i}$ of $D_{6Q-1}$. $D_{2Q-1}^{\mathit{sym}}$ and $D_{6Q-1}^{\mathit{sym}}$ are the $Q$-dimensional subspaces of $D_{2Q-1}$ and $D_{6Q-1}$ spanned by the $E_{i}$ and $F_{i}$ respectively.
\end{definition}

\begin{definition}
\label{def3.4}
$X\rightarrow D_{2Q-1}^{\mathit{sym}}$ is the map taking $[i,j]P^{k}$ to 0 when $k<Q-1$, and to $[i,j]$ when $k=Q-1$ (so that $i+j=2Q-1$).
\end{definition}

\begin{definition}
\label{def3.5}
$Y\rightarrow D_{6Q-1}^{\mathit{sym}}\oplus L$ takes\\
\vspace{-6ex}

\begin{eqnarray*}
[i,j]P^{k} & \mbox{ to } & \left([i,j](\alpha x^{2}y^{2})^{k},0\right)\\
\Delta & \mbox{ to } & (0,1)
\end{eqnarray*}
\end{definition}

Our descriptions of bases of $X$ and $(Y,\Delta)$ show that these $L$-linear maps are well-defined. They are evidently onto.

\begin{lemma}
\label{lemma3.6}
Let $X_{0}$ and $Y_{0}$ be the kernels of the maps of Definitions \ref{def3.4} and \ref{def3.5}. Then $g_{\alpha}$ maps $X_{0}$ bijectively to $Y_{0}$.
\end{lemma}

\begin{proof}
Our description of a basis of $X$ shows that $X_{0}$ is spanned by the $[i,j]P^{k}$ with $i+j+4k=6Q-5$ and $k=0,1,\ldots ,Q-2$. So a non-zero element, $u$, of $X_{0}$ has the form $A(x,y)z^{k}+$ terms of lower degree in $z$, where $A(x,y)\ne 0$ in $D$ and $k<4Q-4$. Then $g_{\alpha}u=A(x,y)z^{k+4}+\cdots \ne 0$; we conclude that $g_{\alpha}$ maps $X_{0}$ injectively. If $k\le Q-2$, then $g_{\alpha}[i,j]\cdot P^{k}=[i,j]P^{k+1}+\alpha[i+2,j+2]P^{k}$. Both terms on the right map to $\left([i,j](\alpha x^{2}y^{2})^{k+1},0\right)$ under the map of Definition \ref{def3.5}, and we conclude that $g_{\alpha}(X_{0})\subset Y_{0}$. Note also that the maps of Definition \ref{def3.4} and \ref{def3.5} are onto, that $\dim X=\dim Y$, and that $\dim D_{2Q-1}^{\mathit{sym}} = \dim D_{6Q-1}^{\mathit{sym}} = Q$.  This tells us that $\dim X_{0}=\dim Y_{0}$, so that $g_{\alpha}\cdot X_{0}=Y_{0}$.
\qed
\end{proof}

In view of Lemma \ref{lemma3.6}, $N_{6Q-5} \cap X$ identifies with the kernel of the map $D_{2Q-1}^{\mathit{sym}} \rightarrow D_{6Q-1}^{\mathit{sym}}\oplus L$ induced by $g_{\alpha}: X\rightarrow (Y,\Delta)$. With respect to the bases $E_{1},\ldots ,E_{Q}$ of $D_{2Q-1}^{\mathit{sym}}$ and $F_{1},\ldots ,F_{Q},1$ of $D_{6Q-1}^{\mathit{sym}}\oplus L$ the matrix of this induced map has the form $\left(\frac{M}{b}\right)$ where $M$ is a $Q$ by $Q$ matrix and $b=(b_{1},\ldots , b_{Q})$ is a row vector. We shall use Theorem \ref{theorem2.5} to write down $M$ and $b$.

\begin{lemma}
\label{lemma3.7}
The map $D_{2Q-1}^{\mathit{sym}} \rightarrow D_{6Q-1}^{\mathit{sym}}\oplus L$ induced by $g_{\alpha}: X\rightarrow (Y,\Delta)$ takes $E_{j}$ to $\left(E_{j}\cdot \left(\sum_{rs=Q}\alpha^{s}C_{r}^{s}x^{2s}y^{2s}\right),b_{j}\right)$, where $b_{j}=0$ if $2j-1\equiv 2Q-2j \mod{3}$, and is 1 otherwise.
\end{lemma}

\begin{proof}
$E_{j}$ pulls back to $E_{j}\cdot P^{Q-1}$ in $X$. Multiplication by $g_{\alpha}$ takes this to $E_{j}\cdot (\alpha x^{2}y^{2}P^{Q-1}+P^{Q})$. By Theorem \ref{theorem2.5} this is $$\textstyle b_{j}\Delta + E_{j}\left(\alpha x^{2}y^{2}P^{Q-1}+\funnysum (C_{r}P)^{s}\right)\!\!.$$ Under the map of Definition \ref{def3.5}, the first term on the right goes to $(0,b_{j})$, while the second goes to $E_{j}\cdot\sum_{rs=Q}(C_{r}\alpha x^{2}y^{2})^{s}$, giving the lemma.
\qed
\end{proof}

\begin{theorem}
\label{theorem3.8}
Situation as in Lemma \ref{lemma3.7}. The image of $E_{j}$ is $\left(\sum \alpha ^{s}F_{i},b_{j}\right)$ where the sum extends over all pairs $(s,i)$ with $s/Q$ and $i\equiv j\mod{3s}$.
\end{theorem}

\begin{proof}
Using the definitions of $A_{r}$ and $C_{r}$ we find that $C_{r}x^{2}y^{2}\! =\! \sum \! x^{2r+2k}y^{2r-2k}$, the sum extending over all $k$ in $(-r,r)$ with $k\equiv 0\mod{3}$. So $C_{r}^{s}x^{2s}y^{2s}=\sum x^{2Q+2l}y^{2Q-2l}$, the sum extending over all $l$ in $(-Q,Q)$ with $l\equiv 0\mod{3s}$. Then $E_{j}(C_{r}^{s}x^{2s}y^{2s})$ is $\sum F_{i}$, the sum extending over all $i\equiv j\mod{3s}$, and Lemma \ref{lemma3.7} gives the result.
\qed
\end{proof}

\begin{corollary}
\label{corollary3.9}
Let $b_{i}=0$ if $2i-1\equiv 2Q-2i\mod{3}$ and $b_{i}=1$ otherwise.  Then the matrix of the induced map $D_{2Q-1}^{\mathit{sym}}\rightarrow D_{6Q-1}^{\mathit{sym}}\oplus L$ with respect to the bases introduced earlier is $\left(\frac{M}{b(Q)}\right)$ where $m_{i,j}=\sum\alpha^{s}$, the sum extending over all $s/Q$ with $i\equiv j\mod{3s}$, and $b(Q)=(b_{1},\ldots ,b_{Q})$.
\end{corollary}

\begin{corollary}
\label{corollary3.10}
Suppose that $\alpha \in L$ is transcendental over $Z/2$. Then\linebreak $N_{6Q-5}\cap X=(0)$.
\end{corollary}

\begin{proof}
The matrix $M$ of Corollary \ref{corollary3.9} has entries in $Z/2[\alpha]$. Each $m_{i,i}$ is a degree $Q$ polynomial in $\alpha$ while the other entries have degree $<Q$. Since $\alpha$ is transcendental over $Z/2$, $\det M\ne 0$, $\left(\frac{M}{b(Q)}\right)$ has rank $Q$ and $D_{2Q-1}^{\mathit{sym}}\rightarrow D_{6Q-1}^{\mathit{sym}}\oplus L$ is 1--1. But the kernel of this map identifies with $N_{6Q-5}\cap X$.
\qed
\end{proof}

For the rest of this section we assume $\alpha$ transcendental over $Z/2$. We'll use Corollary \ref{corollary3.10} to show that $N_{6Q-5}$ is $(0)$. Any $u\ne 0$ in $O$ may be written as $A(x,y)\cdot z^{r}+$ lower degree terms in $z$, where $A(x,y)\ne 0$ in $D$, and $r<4Q$. We say that $u$ has $z$-degree $r$.

\begin{lemma}
\label{lemma3.11}
Suppose that $u\in O_{6Q-5}$ is fixed by $(x,y) \rightarrow (y,x)$ and has $z$-degree $\le 4Q-4$.  Then if $g_{\alpha}u\in g_{\alpha}X$, $u\in X$.
\end{lemma}

\begin{proof}
We argue by induction on the $z$-degree of $u$. If the $z$-degree is 0, then $u$, being fixed by $(x,y)\rightarrow (y,x)$, is a linear combination of $x^{i}y^{j}+x^{j}y^{i}$ with $i+j=6Q-5$, and so is in $X$. If $u=A\cdot z^{4k}+ \cdots$, $k>0$, let $v=u+AP^{k}$. Then the $z$-degree of $v$ is $<4k$, and $g_{\alpha}v\in g_{\alpha}X$. By induction, $v\in X$, and so $u\in X$. Suppose finally that $u=A\cdot z^{r}+\cdots$ with $r\not\equiv 0\mod{4}$ and $r<4Q-4$. Then $g_{\alpha}u=A\cdot z^{r+4}+\cdots $ has $z$-degree that is neither divisible by 4 nor equal to 2. As $g_{\alpha} u\in g_{\alpha}X\subset (Y,\Delta)$, our description given earlier of a basis of $(Y,\Delta)$ shows this to be impossible.
\qed
\end{proof}

\begin{lemma}
\label{lemma3.12}
If $u\in N_{6Q-5}$ has $z$-degree $\le 4Q-4$ then $u=0$.
\end{lemma}

\begin{proof}
The linear automorphism $(x,y,z) \rightarrow (y,x,z)$ of $L[x,y,z]$ fixes $g_{\alpha}$. So the automorphism of $P$ that it induces stabilizes $ N_{6Q-5}$.  Let $\bar{u}$ be the image of $u$ under this automorphism. Lemma \ref{lemma3.11} applied to $u+\bar{u}$ shows that $u+\bar{u}\in X$. Since $u+\bar{u}\in N_{6Q-5}$, Corollary \ref{corollary3.10} shows that $u=\bar{u}$. Applying Lemma \ref{lemma3.11} to $u$ we find that $u\in X$. Another application of Corollary \ref{corollary3.10} completes the proof.
\qed
\end{proof}

\begin{theorem}
\label{theorem3.13}
$ N_{6Q-5}=(0)$.
\end{theorem}

\begin{proof}
Replacing $L$ by a larger field, if necessary, we may assume that $L$ contains some $\omega$ with $\omega^{2}+\omega + 1=0$. We make use of 3 linear automorphisms of $L[x,y,z]$:

\begin{eqnarray*}
\sigma: (x,y,z) & \rightarrow & (x,y,z+x+y)\\
\tau: (x,y,z) & \rightarrow & (x,y,z+\omega x + \omega^{2}y)\\
\rho: (x,y,z) & \rightarrow & (x,y,z+\omega^{2}x + \omega y)
\end{eqnarray*}

Since $P=z(z+x+y)(z+\omega x + \omega^{2}y)(z+\omega^{2}x+\omega y)$, these automorphisms fix $P$ as well as $x$ and $y$. So they fix $g_{\alpha}$, and the automorphisms of $O$ that they induce stabilize $N_{6Q-5}$.

Suppose now that $u = Az^{r}+\cdots$ is an element of $N_{6Q-5}$ of $z$-degree $r$. By Lemma \ref{lemma3.12}, $r=4Q-3$, $4Q-2$ or $4Q-1$. Suppose first that $r=4Q-3$. Then $u^{\sigma}+u=A(x+y)\cdot z^{4Q-4}+\cdots$. Since $A$ is a non-zero element of $D_{2Q-2}$, $A\cdot(x+y)\ne 0$ in $D$. This contradicts Lemma \ref{lemma3.12} applied to the element $u^{\sigma}+u$ of $N_{6Q-5}$. Suppose next that $u=Az^{4Q-2}+Bz^{4Q-3}+\cdots$ has $z$-degree $4Q-2$. Then:

\begin{eqnarray*}
u^{\tau}+u &=& \left(A(\omega x+\omega^{2}y)^{2}+B(\omega x+\omega^{2}y)\right)\cdot z^{4Q-4}+\cdots \\
u^{\rho}+u &=& \left(A(\omega^{2} x+\omega y)^{2}+B(\omega^{2} x+\omega y)\right)\cdot z^{4Q-4}+\cdots 
\end{eqnarray*}

Lemma \ref{lemma3.12} applied to $u^{\tau}+u$ and $u^{\rho}+u$ shows that both are 0. This immediately tells us that $(x^{3}+y^{3})\cdot A$ is 0 in $D$. Since $A$ is a non-zero element of $D_{2Q-3}$ this is impossible. Finally if $u=Az^{4Q-1}+\cdots$ has $z$-degree $4Q-1$, $u^{\sigma}+u=A(x+y)z^{4Q-2}+\cdots$, and we get an element of $N_{6Q-5}$ of $z$-degree $4Q-2$; we've shown this can't happen.
\qed
\end{proof}

\begin{corollary}
\label{corollary3.14}
Let $R_{\alpha}=L[x,y,z]/g_{\alpha}$ where $\alpha \in L$ is transcendental over $Z/2$. Let $f$ be any degree 6 element of $R_{\alpha}$ and $I$ be the ideal $(x^{4},y^{4},z^{4})$ of $R_{\alpha}$. Then $xyf^{Q}\in I^{[Q]}$ for all $Q$. Consequently, $f\in I^{*}$ in $R_{\alpha}$.
\end{corollary}

\begin{proof}
We may assume $Q\! >\! 1$. $O_{12Q-3}$ is 1-dimensional, spanned by $(xyz)^{4Q-1}$. If $i+j=12Q-3$, multiplication gives a bilinear pairing $O_{i}\times O_{j}\rightarrow L$, and one sees immediately that the pairing is non-degenerate. Multiplication by $g_{\alpha}$ gives maps $O_{6Q-5}\rightarrow O_{6Q-1}$ and $O_{6Q-2}\rightarrow O_{6Q+2}$ that are dual under the above pairings. By Theorem \ref{theorem3.13} the first of these maps is 1--1. So the second is onto, and in particular $xyf^{Q}$ lies in its image. In other words, $xyf^{Q}\in (x^{4Q}, y^{4Q}, z^{4Q}, g_{\alpha})$ in $L[x,y,z]$. Passing to $R_{\alpha}$ we get the result.
\qed
\end{proof}

\begin{theorem}
\label{theorem3.15}
Let $L$ be an algebraically closed field of characteristic 2, $P=z^{4}+xyz^{2}+(x^{3}+y^{3})z$ and $P_{1}=x^{2}y^{2}$. Let \rgen be as in section \ref{section1}, $f$ be any degree 6 element of $L[x,y,z]$ and $I$ be the ideal $(x^{4},y^{4},z^{4})$ of $L[x,y,z]$. Then, in the language of Theorem \ref{theorem1.3}, $f\in I^{*}$ in \rgen.
\end{theorem}

\begin{proof}
$\rgen = L(t)[x,y,z]/g_{t}$ and we use Corollary \ref{corollary3.14} with $L$ replaced by $L(t)$.
\qed
\end{proof}

\section{Matrix calculations---the case of algebraic $\bm{\alpha}$}
\label{section4}

\begin{definition}
\label{def4.1}
Suppose $Q\ge 2$ is a power of 2. A matrix $M=|m_{i,j}|$ $1\le i,j\le Q$ with entries in $L$ is a ``special $Q$-matrix'' if the following hold:

\begin{enumerate}
\item[(1)] $m_{i,j}=0$ if $i\not\equiv j\mod{3}$ or\quad $i=j$.
\item[(2)] If $i\equiv j\mod{3}, i\ne j$, then $m_{i,j}\ne 0$, and depends only on $\ord_{2}(i-j)$.
\end{enumerate}
\end{definition}

\begin{theorem}
\label{theorem4.2}
A special $Q$-matrix has rank $Q-2$.
\end{theorem}

\begin{proof}
We argue by induction on $Q$. When $Q=2$,  $M=\left(\rule{0pt}{3ex}\right.\raisebox{.5ex}{$\begin{array}{c@{\hspace{.5em}}c}0&0\\[-1ex]0&0\end{array}$}\left.\rule{0pt}{3ex}\right)$. When $Q\ge 4$, write $M$ as
\vspace{1ex}

\centerline{
\footnotesize
$ \left (%
\begin{array}{c@{\acs}c@{\acs}c}
M_{1} & M_{2} & M_{3}\\[-1ex]
M_{4} & N & M_{5}\\[-1ex]
M_{6} & M_{7} & M_{8}
\end{array}
\right )%
$
}

where $M_{1}$ and $M_{8}$ are $\frac{Q}{4}$ by $\frac{Q}{4}$ matrices.  Using the fact that $M$ is a special $Q$-matrix we find that $M_{1}=M_{8}$, $M_{2}=M_{7}$, $M_{3}=M_{6}$, $M_{4}=M_{5}$, that $M_{1}+M_{3}$ is a non-zero scalar matrix and that $N$ is a special $\frac{Q}{2}$-matrix. So we may write $M$ as
\vspace{2ex}

\centerline{
\footnotesize
$ \left (%
\begin{array}{c@{\acs}c@{\acs}c}
M_{1} & D & M_{3}\\[-1ex]
E & N & E\\[-1ex]
M_{3} & D & M_{1}
\end{array}
\right )%
$ \raisebox{-4ex}{.}
}

Making elementary row and column operations we get:
\vspace{4ex}

\centerline{
\footnotesize
$ \left (%
\begin{array}{c@{\acs}c@{\acs}c}
M_{1} & D & M_{1}+M_{3}\\[-1ex]
E & N & 0\\[-1ex]
M_{1}+M_{3} & 0 & 0
\end{array}
\right )%
$ \raisebox{-4ex}{.}
} %

Since $M_{1}+M_{3}$ is a non-zero scalar, further elementary operations yield:
\vspace{4ex}

\centerline{
\footnotesize
$ \left (%
\begin{array}{c@{\acs}c@{\acs}c}
0 & 0 & M_{1}+M_{3}\\[-1ex]
0 & N & 0\\[-1ex]
M_{1}+M_{3} & 0 & 0
\end{array}
\right )%
$ \raisebox{-4ex}{.}
} %

Then rank $M=$ rank $N+2\left(\frac{Q}{4}\right)$ which is $Q-2$ by the induction assumption.
\qed
\end{proof}

Now let $b(Q)=(b_{1},\ldots , b_{Q})$ be the row vector of Corollary \ref{corollary3.9}; $b_{i}=0$ if $2i-1\equiv 2Q-2i\mod{3}$ and is 1 otherwise. Let $b^{*}(Q)=(b_{1}^{*},\ldots , b_{Q}^{*})$ be defined as follows:  $b_{i}^{*}=1$ if $2i-1\equiv 2Q-2i\mod{3}$ and is 0 otherwise.  In other words, $b_{i}^{*}=1+b_{i}$. We need a modification of Theorem \ref{theorem4.2}.

\begin{theorem}
\label{theorem4.3}
Let $M$ be a special $Q$-matrix. Then the $Q+1$ by $Q$ and $Q+2$ by $Q$ matrices 
\vspace{1ex}

\centerline{
\footnotesize
\hfill
$ \left (%
\begin{array}{c}
M\\[1ex] \hline
b(Q)
\end{array}
\right )%
$
\qquad and \qquad
$ \left (%
\begin{array}{c}
M\\[1ex] \hline
b(Q)\\[1ex] \hline
b^{*}(Q)
\end{array}
\right )%
$
\hfill
}

have rank $Q-1$ and $Q$ respectively.
\end{theorem}

\begin{proof}
Again we argue by induction on $Q$. When $Q=2$, $b(Q)=(1,0)$ and $b^{*}(Q)=(0,1)$. Suppose $Q\ge4$. Write $b(Q)$ as a concatenation $(F_{0}|F_{1}|F_{2})$ where $F_{0}$ and $F_{2}$ have length $Q/4$. Since $b_{i+3}=b_{i}$, $F_{0}=F_{2}$, and one verifies that $F_{1}=b\left(\frac{Q}{2}\right)$. As in the proof of Theorem \ref{theorem4.2} we may write 
\vspace{2ex}

\centerline{
\footnotesize
\hfill
$ \left (%
\begin{array}{c}
M\\[1ex] \hline
b(Q)
\end{array}
\right )%
$
\qquad as \qquad
$ \left (%
\begin{array}{c@{\acs}c@{\acs}c}
M_{1} & D & M_{3}\\[-1ex]
E & N & E\\[-1ex]
M_{3} &D & M_{1}\\[-1ex]
F & b\left(\frac{Q}{2}\right) & F
\end{array}
\right )%
$
\hfill
}

The same elementary row and column operations that were performed in the proof of Theorem \ref{theorem4.2} take this matrix to 
\vspace{4ex}

\centerline{
\footnotesize
$ \left (%
\begin{array}{c@{\acs}c@{\acs}c}
0 & 0 & M_{1}+M_{3}\\[-1ex]
0 & N & 0\\[-1ex]
M_{1}+M_{3} &0 & 0\\[-1ex]
0 & b\left(\frac{Q}{2}\right) & 0
\end{array}
\right )%
$ \raisebox{-6ex}{.}
}

The rank of this matrix is 
\vspace{4ex}

\centerline{
\footnotesize
\hfill
$ \mbox{rank} \left (%
\begin{array}{c}
N\\[1ex] \hline
b\left(\frac{Q}{2}\right)
\end{array}
\right )%
+
2\left(\frac{Q}{4}\right)
$, \hfill
}

which is $Q-1$ by the induction assumption.  The calculation of the rank of 
\vspace{2ex}

\centerline{
\footnotesize
$ \left (%
\begin{array}{c}
M\\[1ex] \hline
b(Q)\\[1ex] \hline
b^{*}(Q)
\end{array}
\right )%
$
}
\vspace{-1ex}

is entirely similar.
\qed
\end{proof}

Suppose now that $\alpha \in L^{*}$ is algebraic over $Z/2$. We attach to $\alpha$ a $Q$ as follows:

\begin{definition}
\label{def4.4}
Write $\alpha = \lambda^{2}+\lambda$, and let $m=m(\alpha )$ be the degree of $\lambda$ over $Z/2$. Then $Q=2^{m-1}$. (Since $\alpha \ne 0$, $\lambda\not\in Z/2$ and $Q\ge 2$.)
\end{definition}

\begin{theorem}
\label{theorem4.5}
Let $Q$ be as in Definition \ref{def4.4}. Then the matrix $M$ of Corollary \ref{corollary3.9} is a special $Q$-matrix.
\end{theorem}

\begin{proof}
$m_{i,i}=\sum_{s/Q}\alpha^{s}=\sum_{s/Q}(\lambda^{s}+\lambda^{2s})=\lambda + \lambda^{2^{m}}$.  As the degree of $\lambda$ over $Z/2$ is $m$, each $m_{i,i}$ is 0.  When $i\not\equiv j\mod{3}$ there are no $s$ such that $i\equiv j\mod{3s}$, and so $m_{i,j}=0$. When $i\equiv j\mod{3}$, $i\ne j$, let $l=1+\ord_{2}(i-j)$. Then $m_{i,j}=\sum_{s/2^{l-1}}\alpha^{s}=\lambda + \lambda^{2^{l}}$. Since $\ord_{2}(i-j)<\ord_{2}(Q)$, $l<m$. Thus $m_{i,j}\ne 0$, and only depends on $l$.
\qed
\end{proof}

\begin{corollary}
\label{corollary4.6}
In the situation of Theorem \ref{theorem4.5}, $N_{6Q-5} \cap X$ is a 1-dimensional space.
\end{corollary}

\begin{proof}
Theorems \ref{theorem4.3} and \ref{theorem4.5} show that the matrix 
\vspace{2ex}

\centerline{
\footnotesize
$ \left (%
\begin{array}{c}
M\\[1ex] \hline
b(Q)
\end{array}
\right )%
$
}

of Corollary \ref{corollary3.9} has rank $Q-1$. So the induced map $D_{2Q-1}^{\mathit{sym}} \rightarrow D_{6Q-1}^{\mathit{sym}}\oplus L$ of the last section has 1-dimensional kernel. As we've seen this kernel identifies with $N_{6Q-5} \cap X$.
\qed
\end{proof}

Now let $u$ be a generator of $N_{6Q-5} \cap X$. Our next goal is to show that the co-efficient of $x^{4Q-2}y^{Q-2}z^{Q-1}$ in $u$ is non-zero.

\begin{lemma}
\label{lemma4.7}
If $u$ is in $X_{0}$, no monomial appearing in $u$ can have the exponent of $z$ equal to $Q-1$.
\end{lemma}

\begin{proof}
It's enough to show that no monomial appearing in $P^{k}$, $0\le k\le Q-2$, can have the exponent of $z$ equal to $Q-1$. Write $k$ as $\sum_{1}^{l}b_{i}$ where the $b_{i}$ are distinct powers of 2. Since $k<Q-1$, $l$ is at most $m-2$. Now $P^{k}$ is the product of $\left(z^{4}+xyz^{2}+(x^{3}+y^{3})z\right)^{b_{i}}$. This is a sum of terms, each of the form $\mbox{(an element of }Z/2[x,y]\mbox{)}\cdot z^{\sum a_{i}b_{i}}$ with each $a_{i}=1$, 2 or 4.  So if the result fails, $Q-1$ is a sum of $m-2$ or fewer powers of 2. Then $Q-1$ is a sum of $m-2$ or fewer distinct powers of 2, which is impossible.
\qed
\end{proof}

\begin{definition}
\label{def4.8}
If $v\in D_{2Q-1}^{\mathit{sym}}$, $\rho(v)$ is the co-efficient of $x^{4Q-2}y^{Q-2}z^{Q-1}$ in a pull-back of $v$ to $X$; by Lemma \ref{lemma4.7} this is independent of the choice of the pull-back.
\end{definition}

Now a pull-back of $E_{i}$ to $X$ is $E_{i}P^{Q-1}=E_{i}\left( (x^{3}+y^{3})z+xyz^{2}+z^{4}\right)^{Q-1}$.  So $\rho (E_{i})$ is the co-efficient of $x^{4Q-2}y^{Q-2}$ in $E_{i}(x^{3}+y^{3})^{Q-1}$. By Lemma \ref{lemma2.6} this is just the $b_{i}^{*}$ defined after Theorem \ref{theorem4.2}.

\begin{theorem}
\label{theorem4.9}
If $u$ is a generator of $N_{6Q-5}\cap X$, the co-efficient of \linebreak $x^{4Q-2}y^{Q-2}z^{Q-1}$ in $u$ is $\ne 0$.
\end{theorem}

\begin{proof}
Combining the map $D_{2Q-1}^{\mathit{sym}} \rightarrow D_{6Q-1}^{\mathit{sym}}\oplus L$ induced by $X\rightarrow (Y,\Delta )$ with $\rho$ we get a map $D_{2Q-1}^{\mathit{sym}} \rightarrow (D_{6Q-1}^{\mathit{sym}}\oplus L)\oplus L$. The discussion above, combined with Theorem \ref{theorem4.5}, shows that with respect to the obvious bases the matrix of this map is

\centerline{
$ \left (%
\begin{array}{c}
M\\[1ex] \hline
b(Q)\\[1ex] \hline
b^{*}(Q)
\end{array}
\right )%
$
}

where $M$ is a special $Q$-matrix. By Theorem \ref{theorem4.3} this matrix has rank $Q$; consequently $D_{2Q-1}^{\mathit{sym}} \rightarrow (D_{6Q-1}^{\mathit{sym}}\oplus L)\oplus L$ is 1--1. The image, $\bar{u}$, of $u$ in $D_{2Q-1}^{\mathit{sym}}$ is $\ne 0$. Since the image of $\bar{u}$ in $D_{6Q-1}^{\mathit{sym}}\oplus L$ is 0, $\rho(\bar{u})\ne 0$, giving the theorem.
\qed
\end{proof}

\begin{theorem}
\label{theorem4.10}
Suppose $\alpha \ne 0$ is algebraic over $Z/2$, $R_{\alpha}=L[x,y,z]/g_{\alpha}$, $f=y^{3}z^{3}$ and $I=(x^{4},y^{4},z^{4})$. Then there is a $Q$ such that $xyf^{Q}\not\in I^{[Q]}$ in $R_{\alpha}$.
\end{theorem}

\begin{proof}
Take $Q$ as in Definition \ref{def4.4}. Let $u$ be as in Theorem \ref{theorem4.9}. Then the co-efficient of $(xyz)^{4Q-1}$ in $uxyf^{Q}$ is the co-efficient of $x^{4Q-2}y^{Q-2}z^{Q-1}$ in $u$, which is $\ne 0$ by Theorem \ref{theorem4.9}. So $uxyf^{Q}\ne 0$ in $O$. Since $g_{\alpha}u=0$, $xyf^{Q} \not\in g_{\alpha}O$. In other words, $xyf^{Q}\not\in (x^{4Q}, y^{4Q}, z^{4Q}, g_{\alpha})$ in $L[x,y,z]$. Now pass to $R_{\alpha}$.
\qed
\end{proof}

\section{Test elements}

\begin{definition}
\label{def5.1}
$c\ne 0$ in $R_{\alpha}$ is a ``test element'' if whenever $J$ is an ideal of $R_{\alpha}$ and $h\in J^{*}$, then $ch^{q}\in J^{[q]}$ for all $q$.
\end{definition}

\begin{remark}
\label{remark1}
Suppose that for each $\alpha\ne 0$, $xy$ is a test element in $R_{\alpha}$. Then the localization problem has a negative solution. To see this, take $L$ algebraically closed of characteristic 2, $P=z^{4}+xyz^{2}+(x^{3}+y^{3})z$, $P_{1}=x^{2}y^{2}$, $I=(x^{4},y^{4},z^{4})$ and $f=y^{3}z^{3}$. We saw in section \ref{section3} that $f\in I^{*}$ in $\rgen$. If $\alpha\ne 0$ in $L$ is algebraic over $Z/2$ then $R_{\alpha}$ is a domain, and Theorem \ref{theorem4.10} shows that $xyf^{Q}\not\in I^{[Q]}$ for some $Q$. Since $xy$ is a test element in $R_{\alpha}$, $f\not\in I^{*}$ in $R_{\alpha}$. As there are infinitely many such $\alpha$, Brenner's Theorem \ref{theorem1.3} gives the result.
\end{remark}

\begin{remark}
\label{remark2}
In fact, $xy$ is a test element in each $R_{\alpha}$. Since each $g_{\alpha}$, $\alpha\ne 0$, defines a smooth plane quartic this is a special case of the following deep result of Brenner---Let $A=L[x,y,z]/g$ where $\charstic L=p$ and $g$ is a form of degree $r$ defining a smooth projective plane curve. Let $J$ be an ideal of $A$ and $h\in J^{*}$. Then if $c\in A$ is homogeneous of degree $> r-3+\frac{r-3}{p}$, $ch^{q}\in J^{[q]}$ for all $q$.
\end{remark}

But the proof of this result is deep, using homological algebra, vector bundle theory and an ampleness criterion of Hartshorne and Mumford. It can't be a part of any short self-contained treatment of our counterexample, and in this exposition I'll take another route. For clarity write $\theta$ for the image of $z$ in $A=R_{\alpha}=L[x,y,z]/g_{\alpha}$, so that $A=L[x,y,\theta]$.

\begin{lemma}
\label{lemma5.2}
For each power $q$ of 2 and each $j$, $(x^{3}+y^{3})^{q-1}\theta ^{j}\in L[x,y,\theta^{q}]$.
\end{lemma}

\begin{proof}
$q=1$ is clear. If $q=2$, we may assume $j=1$. But $(x^{3}+y^{3})\theta=\alpha x^{2}y^{2}+xy\theta^{2}+\theta^{4}$, giving the result.  Taking $q^{\mathrm{th}}$ powers we find that $(x^{3}+y^{3})^{q}\cdot(\theta^{q})^{j}\in L[x,y,\theta^{2q}]$. We can now prove the lemma by induction on $q$: $(x^{3}+y^{3})^{2q-1}\cdot\theta^{j}=(x^{3}+y^{3})^{q}\cdot\left((x^{3}+y^{3})^{q-1}\theta^{j}\right)\in (x^{3}+y^{3})^{q}\cdot L[x,y,\theta^{q}]$. But each $(x^{3}+y^{3})^{q}\cdot (\theta^{q})^{j}$ is in $L[x,y,\theta^{2q}]$.
\qed
\end{proof}

Now assume that $L$ is algebraically closed. The arguments that follow are made working in an algebraic closure of the field $L(x,y,\theta )$.

\begin{lemma}
\label{lemma5.3}
Suppose $d\ne 0$ is in $L[x,y]$. Then for each large power, $r$, of 2 there is an $A$-linear map $\gamma: L[x^{\frac{1}{r}},y^{\frac{1}{r}},\theta^{\frac{1}{r}}]\rightarrow A$ taking $d^{\frac{1}{r}}$ to $x^{3}+y^{3}$.
\end{lemma}

\begin{proof}
$(x^{3}+y^{3})^{r}\cdot \theta^{j}\in L[x,y,\theta^{r}]$. So $(x^{3}+y^{3})\theta^{\frac{j}{r}}\in  L[x^{\frac{1}{r}},y^{\frac{1}{r}},\theta]$, and  $(x^{3}+y^{3})L[x^{\frac{1}{r}},y^{\frac{1}{r}},\theta^{\frac{1}{r}}]\subset L[x^{\frac{1}{r}},y^{\frac{1}{r}},\theta]$. Now the $x^{\frac{i}{r}}y^{\frac{j}{r}}$, $i$ and $j<r$, form a basis of $L[x^{\frac{1}{r}},y^{\frac{1}{r}}]$ over $L[x,y]$. Since $\theta$ is separable over $L[x,y]$ they are also a basis of $L[x^{\frac{1}{r}},y^{\frac{1}{r}},\theta]$ over $L[x,y,\theta]$ (and of $L(x^{\frac{1}{r}},y^{\frac{1}{r}},\theta^{\frac{1}{r}})=L(x^{\frac{1}{r}},y^{\frac{1}{r}},\theta)$ over $L(x,y,\theta)$). We may assume that some monomial appearing in $d$ has co-efficient 1. Since $r$ is large, $d^{\frac{1}{r}}$ is an $L$-linear combination of our basis elements $x^{\frac{i}{r}}y^{\frac{j}{r}}$; also one of the projection maps $p: L(x^{\frac{1}{r}},y^{\frac{1}{r}},\theta)\rightarrow L(x,y,\theta)$ takes $d^{\frac{1}{r}}$ to 1. Let $\gamma$ be the map $u\rightarrow p\left((x^{3}+y^{3})u\right)$. Then $\gamma(d^{\frac{1}{r}})=x^{3}+y^{3}$. Since $(x^{3}+y^{3})\cdot L[x^{\frac{1}{r}},y^{\frac{1}{r}},\theta^{\frac{1}{r}}]\subset L[x^{\frac{1}{r}},y^{\frac{1}{r}},\theta]$, and each of the projection maps takes this last ring into $L[x,y,\theta]=A$, we're done.
\qed
\end{proof}

\begin{lemma}
\label{lemma5.4}
If $L$ is algebraically closed, $x^{3}+y^{3}$ is a test element in $A=R_{\alpha}$.
\end{lemma}

\begin{proof}
Suppose $J$ is an ideal of $R_{\alpha}$ and $h\in J^{*}$. Then $dh^{q}\in J^{[q]}$ for some $d\ne 0$ and all $q$. We may replace $d$ by any $A$-multiple and may assume $d\ne 0$ is in $L[x,y]$. Choose $r$ and $\gamma$ as in Lemma \ref{lemma5.3}.  Then $dh^{qr}\in J^{[qr]}$ and so $d^{\frac{1}{r}}h^{q}\in J^{[q]}\cdot L[x^{\frac{1}{r}},y^{\frac{1}{r}},\theta^{\frac{1}{r}}]$.  Applying $\gamma$ we find that $(x^{3}+y^{3})h^{q}\in J^{[q]}$ for all $q$.
\qed
\end{proof}

In the next section we'll use the elementary Lemma \ref{lemma5.4} in place of Brenner's test element theorem to complete the exposition of the counterexample.

\section{$\bm{H^{2}}$---completion of the proof}

Our goal is:

\begin{lemma}
\label{lemma6.1}
Suppose $L$ is algebraically closed. Let $I$ be the ideal $(x^{4},y^{4},z^{4})$ of $A=R_{\alpha}=L[x,y,z]/g_{\alpha}$, $\alpha \ne 0$. Suppose that $c$ and $f$ are homogeneous elements of $R_{\alpha}$ of degrees 2 and 6. Then if $cf^{Q}\not\in I^{[Q]}$ for some $Q$, $f\not\in I^{*}$.
\end{lemma}

Note that Lemma \ref{lemma6.1} and our earlier results provide the negative solution to the localization problem.  For we may argue as in Remark 1 following Definition \ref{def5.1}, using Theorem \ref{theorem4.10} and Lemma \ref{lemma6.1} to see that $f\not\in I^{*}$ in $R_{\alpha}$ when $\alpha$ is algebraic over $Z/2$.

Let $T$ be the graded $L$-algebra $A/(x^{4Q},y^{4Q})=L[x,y,z]/(x^{4Q},y^{4Q},g_{\alpha})$. We develop a few properties of $T$. Evidently 1, $z$, $z^{2}$ and $z^{3}$ form a basis of $T$ over $L[x,y]/(x^{4Q},y^{4Q})$. So an $L$-basis of $T$ consists of the $x^{i}y^{j}z^{k}$ with $i,j < 4Q$ and $k<4$. In particular, $T_{8Q+1}$ is 1-dimensional, spanned by $(xy)^{4Q-1}z^{3}$. Also, the subspace of $T$ annihilated by $x$ and $y$ is 4-dimensional, spanned by the $(xy)^{4Q-1}z^{k}$, $k=0,1,2,3$. It follows that an element of $T$ is annihilated by $x$, $y$ and $z$ if and only if it lies in $T_{8Q+1}$.

\begin{lemma}
\label{lemma6.2}
If $i+j=8Q+1$ the pairing $T_{i}\times T_{j}\rightarrow L$ induced by multiplication is non-degenerate.
\end{lemma}

\begin{proof}
We show the left kernel is $(0)$, arguing by induction on $j$. The case $i=8Q+1, j=0$ is trivial. Suppose $i<8Q+1$ and $u\in T_{i}$ annihilates $T_{j}$. Then $xu$, $yu$ and $zu$ annihilate $T_{j-1}$. By induction $xu$, $yu$ and $zu$ are 0, and since $i<8Q+1$, $u=0$.
\qed
\end{proof}

For the rest of the section we fix $Q$ with $cf^{Q}\not\in I^{[Q]}$. We shall assume that $f\in I^{*}$ and get a contradiction.

\begin{lemma}
\label{lemma6.3}
There exists a $w$ in $A_{2Q-1}$ with:

\begin{enumerate}
\item[(1)] $z^{4Q}w\in (x^{4Q},y^{4Q})$
\item[(2)] $f^{Q}w\not\in (x^{4Q},y^{4Q})$
\end{enumerate}
\end{lemma}

\begin{proof}
Multiplication by $z^{4Q}$ induces maps $T_{2Q+2}\rightarrow T_{6Q+2}$ and $T_{2Q-1}\rightarrow T_{6Q-1}$. These maps are dual under the pairings of Lemma \ref{lemma6.2}. Since $cf^{Q}\not\in I^{[Q]}$ in $A$, $cf^{Q}$ is not in the image of the first map. So there is a $w$ in the kernel of the second map with $wcf^{Q}\ne 0$ in $T$. Thinking of $w$ as an element of $A_{2Q-1}$ we find that $f^{Q}w\not\in (x^{4Q},y^{4Q})$.  Since $w\rightarrow 0$ in $T_{6Q-1}$, $z^{4Q}w\in (x^{4Q},y^{4Q})$.
\qed
\end{proof}

Now let $K$ be the field of fractions of $A$. $A\left[\frac{1}{xy}\right]$, $A\left[\frac{1}{x}\right]$ and $A\left[\frac{1}{y}\right]$ are $A$-submodules of $K$. Let $H^{2}$ be the quotient module $A\left[\frac{1}{xy}\right]/\left(A\left[\frac{1}{x}\right]+A\left[\frac{1}{y}\right]\,\right)$. ($H^{2}$ is a local cohomology module but we won't use any machinery from that theory.) Note that $H^{2}$ is $Z$-graded; when $u$ is in $A_{l}$, $\frac{u}{x^{i}y^{j}}$ has degree $l-i-j$. Using the fact that 1, $z$, $z^{2}$ and $z^{3}$ are a basis of $A$ over $L[x,y]$ we find:

\begin{enumerate}
\item[(1)] $\frac{1}{x^{i}y^{j}}$, $\frac{z}{x^{i}y^{j}}$, $\frac{z^{2}}{x^{i}y^{j}}$ and $\frac{z^{3}}{x^{i}y^{j}}$, $i,j>0$ are an $L$-basis of $H^{2}$.
\item[(2)] $\frac{u}{x^{i}y^{j}}$ is 0 in $H^{2}$ if and only if $u \in (x^{i},y^{j})$ in $A$.
\end{enumerate}

The map $u\rightarrow u^{2}$, $K\rightarrow K$ stabilizes $A\left[\frac{1}{xy}\right]$, $A\left[\frac{1}{x}\right]$ and $A\left[\frac{1}{y}\right]$ and so induces an additive function $\Phi: H^{2}\rightarrow H^{2}$. $\Phi(H^{2}_{l})\subset H^{2}_{2l}$; furthermore $\Phi(aU)=a^{2}\Phi(U)$. If $q=2^{n}$ we abbreviate $\Phi^{n}(U)$ to $U^{[q]}$.

\begin{lemma}
\label{lemma6.4}
There is a $U$ in $H^{2}_{-1}$, $U\ne 0$, such that $(x^{3}+y^{3})\cdot U^{[q]}=0$ for all $q$.
\end{lemma}

\begin{proof}
Take $w$ as in Lemma \ref{lemma6.3} and let $W$ be the element $\frac{w}{x^{4Q}y^{4Q}}$ of $H^{2}$; set $U=f^{Q}W$.  The degree of $U$ is $(2Q-1)-8Q+6Q=-1$. Since $f^{Q}w\not\in (x^{4Q},y^{4Q})$, $U\ne  0$. Since $x^{4Q}w$, $y^{4Q}w$ and $z^{4Q}w$ are all in $(x^{4Q},y^{4Q})$, $I^{[Q]}\cdot W=(0)$. Applying $\Phi$ repeatedly we find that $I^{[qQ]}W^{[q]}=(0)$.

We are assuming that $f\in I^{*}$. Since $x^{3}+y^{3}$ is a test element in $A$, $(x^{3}+y^{3})f^{qQ}\in I^{[qQ]}$. So $(x^{3}+y^{3})f^{qQ}W^{[q]}=0$. But $f^{qQ}W^{[q]}=U^{[q]}$.
\qed
\end{proof}

\begin{lemma}
\label{lemma6.5}
Suppose $\alpha \ne 1$ and $U$ is a non-zero element of $H^{2}_{-1}$. Then $(x^{3}+y^{3})U^{[8]}\ne 0$.
\end{lemma}

\begin{proofsketch}
Since $U$ has degree $-1$ it is an $L$-linear combination of $\frac{z}{xy}$, $\frac{z^{2}}{x^{2}y}$, $\frac{z^{2}}{xy^{2}}$, $\frac{z^{3}}{x^{3}y}$, $\frac{z^{3}}{x^{2}y^{2}}$ and $\frac{z^{3}}{xy^{3}}$. I'll assume first that $U$ is an $L$-linear combination of $\frac{z}{xy}$, $\frac{z^{2}}{x^{2}y}$ and $\frac{z^{2}}{xy^{2}}$. We know from Lemma \ref{lemma2.2} that $(x^{3}+y^{3})z^{4Q}\equiv (x^{3}+y^{3})\sum_{rs=Q}(C_{r}P)^{s}\bmod{(x^{4Q},y^{4Q})}$ in the polynomial ring $L[x,y,z]$. So, $\bmod{(x^{4Q},y^{4Q},g_{\alpha})}$, $(x^{3}+y^{3})z^{4Q}\equiv (x^{3}+y^{3})\sum_{rs=Q}(C_{r}\alpha x^{2}y^{2})^{s}$. Taking $Q=2$ we get:
$$(x^{3}+y^{3})\left(\frac{z}{xy}\right)^{[8]}=\frac{x^{3}+y^{3}}{x^{8}y^{8}}(\alpha^{2}x^{4}y^{4}+\alpha x^{4}y^{4})=(\alpha^{2}+\alpha)\left(\frac{1}{x^{4}y}+\frac{1}{xy^{4}}\right)$$
Taking $Q=4$ we get:
$$(x^{3}+y^{3})\left(\frac{z^{2}}{x^{2}y}\right)^{[8]}=\frac{x^{3}+y^{3}}{x^{16}y^{8}}(\alpha^{4}x^{8}y^{8}+\alpha^{2}x^{8}y^{8}+\alpha(x^{14}y^{2}+x^{8}y^{8}+x^{2}y^{14}))$$
So 
$$(x^{3}+y^{3})\left(\frac{z^{2}}{x^{2}y}\right)^{[8]}=\frac{x^{3}+y^{3}}{x^{16}y^{8}}\cdot\alpha x^{14}y^{2}=\frac{\alpha}{x^{2}y^{3}}$$
Similarly, 
$$(x^{3}+y^{3})\left(\frac{z^{2}}{xy^{2}}\right)^{[8]}=\frac{\alpha}{x^{3}y^{2}}$$
Since $\alpha^{2}+\alpha \ne 0$, and no $L$-linear combination of $\frac{1}{x^{4}y}+\frac{1}{xy^{4}}$, $\frac{1}{x^{2}y^{3}}$ and $\frac{1}{x^{3}y^{2}}$ can be 0, we're done. When $U$ is an $L$-linear combination of all 6 basis elements of $H^{2}_{-1}$, one may proceed by making a similar but more elaborate calculation. Alternatively one may use the automorphisms $\sigma$, $\tau$ and $\rho$ of Theorem \ref{theorem3.13}, which act on $H^{2}$, to construct a non-zero $V$ in $H^{2}_{-1}$, with $V^{[8]}=0$, which is an $L$-linear combination of $\frac{z}{xy}$, $\frac{z^{2}}{x^{2}y}$ and $\frac{z^{2}}{xy^{2}}$. We leave details to the reader. 
\end{proofsketch}

We can now complete the proof of Lemma \ref{lemma6.1}; we cannot simultaneously have $f\in I^{*}$ and $cf^{Q}\not\in I^{[Q]}$ for some $Q$. If $\alpha \ne 1$ this follows from Lemmas \ref{lemma6.4} and \ref{lemma6.5}. If $\alpha = 1$ we modify the proof of Lemma \ref{lemma6.5} to show that $(x^{3}+y^{3})U^{[16]}\ne 0$, once again contradicting Lemma \ref{lemma6.4}.

% \appendix

% \section{}
% \label{}

\end{document}